\documentclass[reqno]{amsart}
\usepackage{graphicx, cite, amssymb}
\usepackage[utf8]{inputenc}
\usepackage[english]{babel}
\newtheorem{thm}{Theorem}
\newtheorem{cor}{Corollary}
\newtheorem{lem}{Lemma}
\newtheorem{prop}{Proposition}
\theoremstyle{definition}

\theoremstyle{remark}
\newtheorem{rem}{Remark}

\newcommand{\supp}{\mathop{\rm supp}}
\newcommand{\Real}{\mathbb R}

\renewcommand{\kappa}{\varkappa}
\newcommand{\lm}{\lambda}

\newcommand{\cH}{\mathcal{H}}
\newcommand{\cI}{\mathcal{I}}
\newcommand{\dmn}{\mathop{\rm dom}}

\begin{document}

\title[Eigenvalues of Schr\"{o}dinger operators near thresholds]{
Eigenvalues of Schr\"{o}dinger operators near thresholds: two term approximation}
\author{Yuriy Golovaty}%
\address{Department of Mechanics and Mathematics,
  Ivan Franko National University of Lviv\\
  1 Universytetska str., 79000 Lviv, Ukraine}
\curraddr{}
\email{yuriy.golovaty@lnu.edu.ua}


\begin{abstract}
We consider one dimensional Schr\"{o}dinger operators
\begin{equation*}
H_\lm=-\frac{d^2}{dx^2}+U+ \lm V_\lm
\end{equation*}
with nonlinear dependence on  the parameter $\lm$ and study the small $\lm$ behaviour of eigenvalues. Potentials $U$ and $V_\lm$ are real-valued bounded functions of compact support. Under some assumptions on $U$ and $V_\lm$, we prove  the existence of a negative eigenvalue that is absorbed at the bottom of the continuous spectrum as $\lm\to 0$.  We also construct two-term asymptotic formulas for the threshold  eigenvalues.
\end{abstract}

\maketitle

\section{Introduction}
About forty years ago, Simon and Klaus \cite{Simon:1976,Klaus:1977, KlausSimonAnPh:1980, KlausSimonComMathPh:1980} started studying the low energy behaviour of the so-called weakly coupled Hamiltonians $-\Delta+\lm V$. The considerable interest  has been in the study of negative-energy bound states and their small $\lm$ behaviour, as well as in the study of the absorption of the  eigenvalues by the continuous spectrum. The main results here have been concerned with Schr\"{o}dinger operators in one and two dimensions, because  in three dimensions the weakly coupled Hamiltonians have no bound state if $\lambda$ is small enough, i.e., if potential $\lambda V$ is a sufficiently shallow well.
For the case of  1D Hamiltonians $\cH_\lm=-\frac{d^2}{dx^2}+\lm V$,
an suitable short-range potential $V$  can produce a bound state for all small $\lm$. Assuming that  $V$ is different from zero and $\int_{\Real}(1+|x|^2)|V(x)|\,dx<\infty$, Simon~\cite{Simon:1976} proved that the operator $\cH_\lm$ has a negative-energy bound state $e_\lm$ for all small positive $\lm$ if and only if  $\int_{\Real}V(x)\,dx\leq0$.
If $\cH_\lm$ does have an eigenvalue, then it is unique and simple, and obeys
\begin{equation}\label{AbCaGoldFormula}
  \sqrt{-e_\lm}=-\frac{\lm}{2}\int_{\Real}V(x)\,dx
  -\frac{\lm^2}{4}\iint_{\Real^2}V(x)\,|x-y|\,V(y)\,dx\,dy+o(\lm^2)
\end{equation}
as $\lm\to 0$. This asymptotic formula is due to Abarbanel, Callan and Goldberger, but it was not published by them; \eqref{AbCaGoldFormula} was firstly announced  by Simon \cite{Simon:1976}.
The eigenvalue $e_\lm$ approaches zero as $\lambda$ goes to zero and it is absorbed in the limit at the bottom of the continuous spectrum $[0,+\infty)$. Then we say that $\lm=0$ is a coupling constant threshold for $\cH_\lm$.
Klaus \cite{Klaus:1977} has  extended this result to the class of potentials $V$ obeying the condition $\int_{\Real}(1+|x|)|V(x)|\,dx<\infty$.

In \cite{Simon:1977, Klaus:1982}, the threshold behaviour has been studied as a general perturbation phenomenon and some general results on existence and asymptotic behaviour of eigenvalues for self-adjoint operators $A+\lm B$ have been obtained. The main tool was the so-called Birman-Schwinger principle. Klaus \cite{Klaus:1982} has also applied these results to several special cases. One of them has been concerned with the Hamiltonian $-\frac{d^2}{dx^2}+U+\lm V$.  If a certain  relation between the potentials $U$ and $V$ holds, then the operator has a small negative-energy bound state (not necessarily a unique one) in the limit of weak coupling.
Namely, it has been proved that the operator has the coupling constant threshold $\lm=0$,
if the unperturbed ope\-ra\-tor $-\frac{d^2}{dx^2}+U$ possesses a zero-energy resonance with a half-bound state $u$ and $\int_\Real Vu^2\,dx<0$.
Among the negative eigenvalues  there exists  only one that is absorbed by the continuous spectrum as $\lm\to 0$. A unique threshold eigenvalue $e_\lm$ is analytic at $\lm=0$ and obeys
\begin{equation}\label{KlausFormula}
 \sqrt{-e_\lm}=-\frac{\lm}{u_-^2+u_+^2}\int_{\Real}Vu^2\,dx
  +O(\lm^2)
\end{equation}
as $\lm\to 0$, where $u_\pm=\lim\limits_{x\to\pm\infty}u(x)$.
If $\int_\Real Vu^2\,dx=0$ and the support of $V$ lies between two consecutive zeros of $u$, then there exists a bound state near zero for all small enough $\lm$  (positive and negative). Finally, if $\int_\Real Vu^2\,dx>0$, then the operator has no bound state and therefore
$\lm=0$ is not a coupling constant threshold.
We will give the precise definitions of the zero-energy resonances, half-bound states, and coupling constant threshold in the next section.

One of the motivations for writing this article was the desire to  improve  approximation \eqref{KlausFormula}. As another motivation for investigating the threshold behaviour of eigenvalues, we mention  app\-lications
of  this phenomenon to the study of the stability of solutions for the Korteweg-de~Vries equation \cite{ScharfWreszinski:1981} and the existence of 'breathers' (the localized periodic solutions) for discrete nonlinear Schr\"{o}dinger systems \cite{Weinstein:1999, KevrekidisRasmussenBishop2001}.

In this paper, we consider a more general class of  Schrodinger operators
\begin{equation}\label{Hlambda}
H_\lm=-\frac{d^2}{dx^2}+U+ \lm V_\lm, \qquad\dmn H_\lm=W_2^2(\Real)
\end{equation}
with  nonlinear dependence on the positive parameter $\lambda$. We
analyse the existence of negative eigenvalues and their threshold behaviour. Here $U$ and $V_\lm$ are functions of compact support and  $V_\lm=V+\lambda V_1+o(\lambda)$ as $\lm\to 0$.
The spectrum  of $H_\lm$ consists of the essential spectrum $[0, \infty)$  and possibly a finite number of negative eigenvalues. Under certain conditions on the potentials $U$, $V$  and $V_1$  the  operator $H_\lm$  has   a negative eigenvalue $e_\lm$ that is absorbed at the bottom of the essential spectrum as $\lm$ goes to zero. The threshold eigenvalue  may or may not be the ground state. We examine  the asymptotic behaviour of $e_\lm$ as $\lm\to0$ and compute the two term asymptotic formula which in particular improves the approximation \eqref{KlausFormula}. For the case $U=0$ and $V_\lm=V$, our asymptotics turns into the Abarbanel-Callan-Goldberger formula.

The threshold behaviour of eigenvalues for operators
$-\frac{d^2}{dx^2}+U+ \lambda\alpha_\lambda V(\alpha_\lambda \cdot)$, where
the positive sequence $\alpha_\lambda$ converges to a finite or infinite limit as $\lambda\to 0$, has recently been studied in \cite{GolovatyThresholds:2019}.
These results  gives us  an example of the non-analytic threshold behaviour  of negative eigenvalues.

The question of how negative eigenvalues are absorbed in the bottom of the essential spectrum  has been discussed by many authors \cite{BlankenbeclerGoldbergerSimon1977, Rauch1980, AlbeverioGesztesyHoegh-Krohn:1982, Holden1985, BolleGesztesyWilk:1985, GesztesyHolden:1987, JensenMelgaard2002, Gadylshin2002, Gadylshin2004, BorisovGadylshin2006, AlbeverioNizhnik:2003}. The Hamiltonians with periodic potentials perturbed by short range ones and the threshold pheno\-me\-na in gaps of the continuous spectrum were studied in
\cite{Klaus:1982, GesztesySimon1993, FassariKlaus1998}.

\section{Main Results}\label{SecMainRes}
We start with some definitions.
Let $A$ and $B_\lm$ be self-adjoint operators and $B_\lm$ be relatively $A$-com\-pact for all $\lm>0$; then $\sigma_{ess}(A+B_\lm)=\sigma_{ess}(A)$.  Suppose that the interval $(a,b)$ is a gap in the spectrum of $A$. If  we can find an eigenvalue $e_\lm$ of $A+B_\lm$ in  $(a,b)$ for all $\lm>0$ with the property that $e_\lm\to a$ or $e_\lm\to b$  as  $\lm\to 0$, then we call $\lm=0$ the \textit{coupling constant threshold}. So the eigenvalue $e_\lm$ is absorbed by the continuous spectrum at  ``time'' $\lm=0$.

We say  operator~$-\frac{d^2}{dx^2}+U$  possesses  a \emph{zero-energy resonance} if there exists a non trivial solution~$u$ of the equation
\begin{equation}\label{EqnHBSu}
 -u'' + Uu= 0
\end{equation}
that is bounded on the whole line. We then call $u$ the \emph{half-bound state}.
Any half-bound state $u$ possesses finite limits $\lim\limits_{x\to\pm\infty}u(x)$, because $u$ is constant outside the support of $U$; both the limits  are different from zero.
Since a half-bound state is defined up to a scalar multiplier, we  say  a half-bound state $u$ is \textit{normalized} if $\lim\limits_{x\to-\infty}u(x)=1$. Let $\theta$ hereafter denote the limit of the  normalized half-bound state as $x\to+\infty$, i.e.,
$\theta:=\lim_{x\to+\infty}u(x)$.
We also introduce the function
\begin{equation*}
\Theta(x)=
\begin{cases}
  1 & \text{if } x<0,\\
  \theta& \text{if } x>0.
\end{cases}
\end{equation*}
Assume $u_1$ is a solution of \eqref{EqnHBSu} such that $u_1(x)=x$ to the left of the support of $U$. Then $u$ and $u_1$
are linearly independent solutions of \eqref{EqnHBSu} and we will show below that there exists a constant $\theta_1$ such that  $u_1(x)=\theta^{-1}x+\theta_1$ for all $x$ large enough (see Fig.~\ref{FigPlots}).
Let $v_*$ be a solution  of  $-v''+Uv=-Vu$ which vanishes to the left of the supports of $U$ and $V$.

\begin{figure}[h]
  \centering
  \includegraphics[scale=0.8]{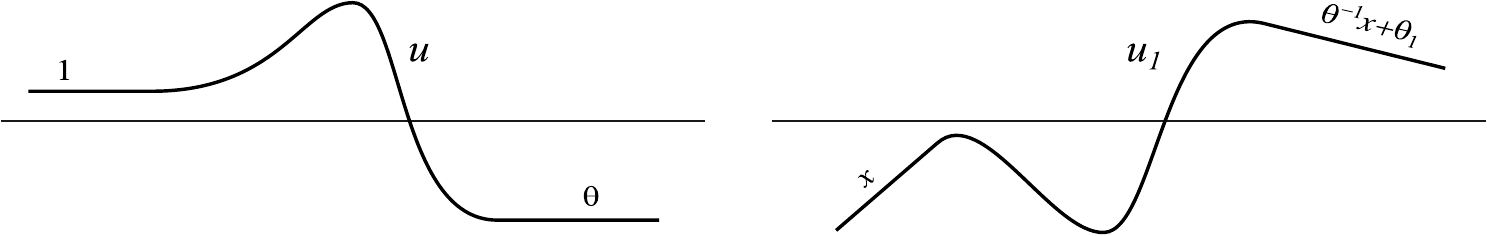}\\
  \caption{Plots of  normalized half-bound state $u$ and solution $u_1$}\label{FigPlots}
\end{figure}

Here and subsequently, $\|\,\cdot\,\|$ stands for the norm in $L_2(\Real)$.

\begin{thm}\label{Theorem1}
Suppose that $U$, $V$ and $V_1$ are functions of compact support belonging to $L^\infty(\Real)$, and $\|V_\lm-V-\lm V_1\|=o(\lm)$ as $\lm\to 0$.
Assume  operator $-\frac{d^2}{dx^2}+U$ has a zero-energy resonance with normalized half-bound state $u$. If
\begin{equation}\label{Omega0isPositive}
  \int_\Real V u^2\,dx<0,
\end{equation}
then operator $H_\lm=-\frac{d^2}{dx^2}+U+ \lm V_\lm$ possesses the coupling constant threshold $\lm=0$, i.e., for all small positive $\lm$ there exists  a negative eigenvalue $e_\lm$ of $H_\lm$ such that $e_\lm\to 0$ as $\lm\to0$. Moreover the threshold eigenvalue $e_\lm$ has  the asymptotic expansion $e_\lm=-\lm^2\left(\omega_0+\omega_1\lm+o(\lm)\right)^2$
as $\lm\to0$, where
\begin{gather}\label{Omega0}
  \omega_0=\frac{1}{\theta^2+1}\int_{\Real} V  u^2\,dx,
\\
\begin{aligned}\label{Omega1}
  \omega_1=
  \frac{1}{\theta^2+1}
  \bigg(
\int_{\Real} V \big( v_*&+\omega_0(\theta^2-1)u_1\big) u\,dx
  \\
  &+
  \omega_0^2 \int_{\Real}(u^2-\Theta^2)\,dx-\omega_0^2\theta^{3}\theta_1
  +
  \int_{\Real} V_1u^2\,dx
\bigg).
\end{aligned}
\end{gather}
\end{thm}

The threshold phenomenon is also possible if inequality \eqref{Omega0isPositive} turns into the equa\-li\-ty. In this case the absorption of the  eigenvalue at the bottom of $\sigma_{ess}(H_\lm)$ occurs  with the rate $O(\lm^4)$ as $\lm\to0$.

\begin{thm}\label{Theorem2}
Under the assumptions of Theorem~\ref{Theorem1}, we suppose that
\begin{equation}\label{Omega0isZero}
  \int_\Real V u^2\,dx=0.
\end{equation}
Then the operator $H_\lm$ has the coupling constant threshold $\lm=0$, if
\begin{equation}\label{betaIntVdiff}
\int_{\Real}(V v_*+ V_1u) u\,dx<0.
\end{equation}
Moreover the threshold eigenvalue $e_\lm$ admits the asymptotics
\begin{equation*}
  e_\lm=-\frac{\lm^4}{(\theta^2+1)^2}\bigg(
  \int_{\Real}V v_* u\,dx
  +
  \int_{\Real} V_1u^2\,dx
\bigg)^2+o(\lm^4)\quad\text{as } \lm\to 0.
\end{equation*}
\end{thm}

Return now to operator family $-\frac{d^2}{dx^2}+U+\lm V$ studied in \cite{Klaus:1982}.

\begin{cor}\label{Cor1}
  Assume  the operator $-\frac{d^2}{dx^2}+U$ has a zero-energy resonance with half-bound state $u$. If
\begin{equation*}
  \int_\Real V u^2\,dx<0,
\end{equation*}
then  $-\frac{d^2}{dx^2}+U+ \lm V$ possesses the coupling constant threshold $\lm=0$ and a negative eigenvalue $e_\lm$ admits the asymptotics
\begin{equation}\label{AsymptoticsCor1}
  e_\lm=-\lm^2(\omega_0+\lm\omega_1+o(\lm))^2,
\end{equation}
where $\omega_0$ is given by  \eqref{Omega0} and
\begin{equation*}
  \omega_1=
  \frac{1}{\theta^2+1}
  \bigg(
\int_{\Real} V \big( v_*+\omega_0(\theta^2-1)u_1\big) u\,dx
  +
  \omega_0^2 \int_{\Real}(u^2-\Theta^2)\,dx-\omega_0^2\theta^{3}\theta_1
\bigg).
\end{equation*}
If $V$ is different from zero and
\begin{equation}\label{Vu0Cor1}
  \int_\Real V u^2\,dx=0,
\end{equation}
then the operator $-\frac{d^2}{dx^2}+U+ \lm V$ has a negative eigenvalue $e_\lm$ with the asymptotics
\begin{equation}\label{ElmCor1Vu0}
  e_\lm=-\frac{\lm^4}{(u_-^2+u_+^2)^2}\left(\iint_{\Real^2} V(x)u(x)\mathcal{E}_U(x-y) V(y)u(y)\,dx\,dy+o(1)\right)^2,
\end{equation}
where $\mathcal{E}_U$ is the fundamental solution for  $\frac{d^2}{dx^2}-U$ which vanishes to the left of $\supp U$.
\end{cor}
\begin{proof}
Most of the proof follows from the previous theorems, assuming $V_\lm=V$ for all $\lm$. We are left with the task of deriving \eqref{ElmCor1Vu0}. If \eqref{Vu0Cor1} holds, then $\omega_0=0$ and
\begin{equation}\label{Omega1forCor1}
   \omega_1=
   \frac{1}{\theta^2+1}
\int_{\Real} V  v_* u\,dx.
\end{equation}
Recall that $v_*$ solves equation $v_*''-Uv_*=Vu$ and vanishes to the left of the supports of $U$ and $V$.
Then $v_*$  can be represented as the convolution
$\mathcal{E}_U*(Vu)$. Hence
\begin{multline}\label{IntVstarCor1}
  \int_{\Real} V v_*u\,dx=
  \int_{\Real} V(x)u(x) (\mathcal{E}_U*Vu)(x)\,dx\\
  =\iint_{\Real^2} V(x)u(x) \mathcal{E}_U(x-y) V(y)u(y)\,dx\,dy.
 \end{multline}
Substituting \eqref{IntVstarCor1} into   \eqref{Omega1forCor1} finishes up the proof.
\end{proof}

\begin{rem}
 Klaus did not use the notion  of a normalized half-bound state. To agree the asymptotic formulas, we rewrite $\omega_0$ and $\omega_1$ in \eqref{AsymptoticsCor1}  in terms of an arbitrary half-bound state $u$ for which $\lim\limits_{x\to\pm\infty}u(x)=u_\pm$.
  Then in notation of \cite{Klaus:1982} we obtain
\begin{gather*}
  \omega_0=\frac{1}{u_-^2+u_+^2}\int_{\Real} V  u^2\,dx,
\\
\begin{aligned}
  \omega_1=
  \frac{1}{u_-^2+u_+^2}
  \bigg(u_-
\int_{\Real} V \bigg( v_*+\frac{\omega_0(u_+^2-u_-^2)}{u_-^2}&\, u_1\bigg) u\,dx
  \\
  & +
  \omega_0^2 \int_{\Real}(u^2-\Xi^2)\,dx-\omega_0^2\theta_1 \frac{u_+^3}{u_-}
\bigg),
\end{aligned}
\end{gather*}
where $\Xi(x)=u_-$  for $x<0$ and $\Xi(x)=u_+$ for $x>0$.
\end{rem}

Let us compare our results with those of Simon when the unperturbed operator is the free Schr\"{o}dinger operator.

\begin{cor}\label{Cor2}
Assume that $U=0$. If the mean value of $V$ is negative, i.e.,
 \begin{equation}\label{Omega0isPositiveU0}
  \int_\Real V\,dx<0,
 \end{equation}
then  $H_\lm=-\frac{d^2}{dx^2}+\lm V_\lm$
has a negative eigenvalue of the form
\begin{equation*}
  e_\lm=-\lm^2(\omega_0+\lm\omega_1+o(\lm))^2
\end{equation*}
as $\lm$ tends to zero, where
\begin{equation}\label{Omega0Omega1ForU0}
  \omega_0=\frac{1}{2}\int_{\Real}V\,dx,
  \qquad
  \omega_1=\frac{1}{4}\iint_{\Real^2} V(x)\,|x-y|\, V(y)\,dx\,dy+
   \frac{1}{2}\int_{\Real} V_1\,dx.
\end{equation}
In the case $V_\lm=V$, this asymptotic formula coincides with  \eqref{AbCaGoldFormula}.
\end{cor}
\begin{proof}
  The trivial potential $U=0$ has a zero-energy resonance with  half-bound state $u=1$; then $\theta=1$ and  $\Theta(x)=1$ for all $x\in \Real$.
 In addition, we have $\theta_1=0$, because  equation $u''=0$ possesses the solution $u_1=x$.
  Therefore condition \eqref{Omega0isPositive} becomes \eqref{Omega0isPositiveU0}, and \eqref{Omega0}, \eqref{Omega1} simplify to read
  \begin{equation*}
    \omega_0=\frac{1}{2}\int_{\Real}V\,dx,
   \qquad
   \omega_1=
  \frac{1}{2}
  \int_{\Real} V  v_*\,dx
    +
  \frac{1}{2}\int_{\Real} V_1\,dx.
  \end{equation*}
The fundamental solution $\mathcal{E}_0(x)=\tfrac12(|x|+x)$
for the differential operator $\frac{d^2}{dx^2}$ va\-ni\-shes
for $x<0$. As in Corollary~\ref{Cor1},  we derive
\begin{multline*}
   \int_{\Real} V(x) v_*(x)\,dx=
  \int_{\Real} V(x) (\mathcal{E}_0*V)(x)\,dx\\
  =\frac{1}{2}\iint_{\Real^2} V(x)\,|x-y|\, V(y)\,dx\,dy
  +\frac{1}{2}\iint_{\Real^2} V(x)\,(x-y)\, V(y)\,dx\,dy \\
  =\frac{1}{2}\iint_{\Real^2} V(x)\,|x-y|\, V(y)\,dx\,dy,
\end{multline*}
because $\iint_{\Real^2} f(x)\,(x-y)\, f(y)\,dx\,dy=0$ for any  $f$, for which the integral exists.
This gives  the second equality in  \eqref{Omega0Omega1ForU0}, and the proof is complete.
\end{proof}

\begin{cor}\label{Cor3}
Assume that $U=0$ and $V$ is different from zero. If
 \begin{equation*}
  \int_\Real V\,dx=0,
 \end{equation*}
then for all nonzero $\lm$, positive or negative, the operator $H_\lm=-\frac{d^2}{dx^2}+\lm V_\lm$
possesses an eigenvalue $e_\lm$ having the asymptotics
\begin{equation}\label{OmegaLmForU0meanV0}
   e_\lm=-\frac{\lm^4}{16} \left( \iint_{\Real^2}V(x)\,|x-y|\,V(y)\,dx\,dy+o(1)\right)^2
\end{equation}
as $\lm\to 0$. This asymptotic formula can be also written in the form
\begin{equation}\label{OmegaLmForU0meanV0}
   e_\lm=-\frac{\lm^4}{4}
   \left( \int_{\Real} \left(
   \int_{-\infty}^x V(y)\,dy\right)^2 dx+o(1)\right)^2.
\end{equation}
\end{cor}

This assertion will be proved in Section~\ref{SecProofs}.

\section{Preliminaries }
We first record some technical facts. Assume, without
loss of generality,   the supports of potentials $U$ and $V_\lm$ lie within $\cI=(-\ell,\ell)$ for $\lm$ small enough. Then a half-bound state of operator $-\frac{d^2}{dx^2}+U$ is  constant outside $\cI$ and its restriction to $\cI$ is a non-trivial solution of  the problem
\begin{equation*}
     - u'' +Uu= 0, \quad t\in \cI,\qquad u'(-\ell)=0, \quad u'(\ell)=0.
\end{equation*}
Moreover,  if $u$ is the normalized half-bound state, then $u(-\ell)=1$ and  $u(\ell)=\theta$.

\begin{prop}\label{PropVpEll}
Assume that $h$ belongs to $L_2(\cI)$ and $\gamma$ is a real number.
 Let $w$ be a solution of the Cauchy problem
  \begin{equation}\label{NeumanProblemV}
     - w'' +Uw= h, \quad t\in \cI,\qquad w(-\ell)=0, \quad w'(-\ell)=\gamma.
\end{equation}
If  $-\frac{d^2}{dx^2}+U$ has a zero-energy resonance with normalized half-bound state $u$, then
\begin{equation}\label{thetaWpr}
 \theta w'(\ell)=\gamma-\int_{-\ell}^\ell hu\,dx.
\end{equation}
 In addition, this solution obeys the estimate
$$
    \|w\|_{C^1(\cI)}\leq C(|\gamma|+\|h\|_{L_2(\cI)})
$$
for some positive $C$ being independent of $\gamma$ and $h$.
\end{prop}
\begin{proof}
 Since $u(-\ell)=1$ and $u(\ell)=\theta$, \eqref{thetaWpr} can be easily  obtained by multiplying the equation in \eqref{NeumanProblemV} by $u$ and integrating by parts.
Next, application of the variation of parameters method  yields
     \begin{equation}\label{CPSolRepresentation}
        w(x)=\gamma (u_1(x)+ \ell u(x))+\int_{-\ell}^x k(x,s)h(s)\,ds,
     \end{equation}
where $k(x,s)=u(x)u_1(s)-u(s)u_1(x)$.
Under the assumptions made on potential $U$, $u$ and $u_1$ belong to $W_2^2(\cI)$; consequently $u, u_1\in C^1(\cI)$ by the Sobolev embedding theorem.
From this and the representation of the first derivative
     \begin{equation*}
         w'(x)=\gamma (u_1'(x)+ \ell u'(x))+\int_{-\ell}^x \frac{\partial k}{\partial x}(x,s)h(s)\,ds
     \end{equation*}
we have $|w(x)|+|w'(x)|\leq |\gamma|(\|u_1\|_{C^1(\cI)}+|\ell|\|u\|_{C^1(\cI)})+c_1\,\|k\|_{C^1(\cI\times \cI)}\|h\|_{L_2(\cI)}\leq C(|\gamma|+\|h\|_{L_2(\cI)})$
for $x\in \cI$, which completes the proof.
\end{proof}

\begin{prop}
  Let $u_1$ be the solution of \eqref{EqnHBSu} as described in Section~\ref{SecMainRes}. Then for some constant $\theta_1$ we have $u_1(x)=\theta^{-1}x+\theta_1$ for all $x>\ell$.
\end{prop}
\begin{proof}
The function $v=u_1+\ell u$ solves the Cauchy problem
  \begin{equation*}
     - v'' +Uv= 0, \quad t\in \cI,\qquad v(-\ell)=0, \quad v'(-\ell)=1
\end{equation*}
and therefore $u_1'(\ell)=\theta^{-1}$ by \eqref{thetaWpr}. Hence $u_1(x)=\theta^{-1}x+\theta_1$
for some $\theta_1$ and all $x>\ell$, which  is the desired conclusion.
\end{proof}

Our method is different from that of Simon and Klaus. We don't use the Birman-Schwinger principle. To prove the main results, we use the asymptotic method of quasimodes or in other words of “almost” eigenvalues and eigenfunctions.
Let $A$ be a self-adjoint operator in a Hilbert space $L$.
We say a pair $(\mu, \phi)\in \Real\times \dmn A$ is a \textit{quasimode} of  $A$ with accuracy $\delta$, if $\|\phi\|_L=1$ and $\|(A-\mu I)\phi\|_L\leq\delta$.

\begin{lem}[\hglue-0.1pt{\cite[p.139]{PDEVinitiSpringer}}]\label{LemQuasimodes}
   Assume $(\mu, \phi)$ is a quasimode of $A$ with accuracy $\delta>0$  and  the spectrum of $A$ is discrete in  the interval
$[\mu-\delta, \mu+\delta]$. Then there exists an eigenvalue $\lambda$ of  $A$ such that $|\lambda-\mu|\leq\delta$.
\end{lem}
\begin{proof}
If $\mu\in \sigma(A)$, then $\lm=\mu$. Otherwise  the distance $d_\mu$ from $\mu$ to the spectrum of $A$  can be computed as
\begin{equation*}
  d_\mu=\|(A-\mu I)^{-1}\|^{-1}
  =\inf_{\psi\neq0}\frac{\|\psi\|_L}{\|(A-\mu I)^{-1}\psi\|_L},
\end{equation*}
where $\psi$ is an arbitrary vector of $L$. Taking $\psi=(A-\mu I)\phi$, we deduce
\begin{equation*}
  d_\mu\leq \frac{\|(A-\mu I)\phi\|_L}{\|\phi\|_L}\leq \delta,
\end{equation*}
from which the assertion  follows.
\end{proof}

\section{Proof of Main Results}\label{SecProofs}

\subsection{Proof of Theorem~\ref{Theorem1}}
In order to prove the existence of a negative eigenvalue for $H_\lm$, we will construct a quasimode $(-\omega^2_\lm, \phi_\lm)$ of  $H_\lm$ as follows.
Suppose that   $-\frac{d^2}{dx^2}+U$ has a zero-energy resonance with normalized half-bound state $u$. We assume $\omega_\lm=\lm(\omega_0+\lm\omega_{1,\lm}+\lm^2\omega_{2,\lm})$ and $\phi_\lm=\psi_\lm/\|\psi_\lm\|$,
where
 \begin{equation}\label{QuasimdePsi}
   \psi_\lm(x)=
   \begin{cases}
     e^{-\omega_\lm(x+\ell)}& \text{for } x<-\ell,\\
     u(x)+\lm v_1(x)+\lm^2 v_{2,\lm}(x)+\lm^3 v_{3,\lm}(x)  & \text{for } |x|<\ell,\\
     a_\lm\,e^{\omega_\lm(x-\ell)}+b_\lm\rho(x-\ell)& \text{for } x>\ell.
   \end{cases}
 \end{equation}
 The functions $v_1$, $v_{2,\lm}$ and $v_{3,\lm}$ are  solutions of the problems
 \begin{align}\label{CPv1}
   &-v_1''+Uv_1=- V u, \qquad v_1(-\ell)=0, \;\; v'_1(-\ell)=-\omega_0;
 \\ \label{CPv2}
   &
   \begin{cases}
     -v_2''+Uv_2=- V v_1-(V_1+g_\lm )u,\\
     \phantom{-}v_2(-\ell)=0, \;\; v_2'(-\ell)=-\omega_{1,\lm}
   \end{cases}
\\ \label{CPv3}
   &
      -v_3''+Uv_3=-f_{3,\lm},\qquad
     \phantom{-}v_3(-\ell)=0, \;\; v_3'(-\ell)=-\omega_{2,\lm}
\end{align}
respectively. Here we set $g_\lm=\lm^{-1}(V_\lm-V-\lm V_1)$
and
\begin{equation*}
  f_{3,\lm}= V v_{2,\lm}+(V_1+\omega_0^2+g_\lm)v_1+ 2\omega_0\omega_{1,\lm}u.
\end{equation*}
We also presume that  $\omega_{1,\lm}$ and $\omega_{2,\lm}$ have finite limits as $\lm\to 0$.
The  function $\rho$ is smooth in $\Real\setminus\{0\}$, $\rho(x)=0$ for $x\leq 0$ and $x\geq 1$, and  $\rho'(+0)=1$. In addition, $\rho''$ is bounded in $[0,1]$.  Hence $\rho$ is  continuous at $x=0$, but the first derivative $\rho'$ has the unit jump at this point. This function  corrects the discontinuity of $\psi'_\lm$ at $x=\ell$.

Let us first show that constants $\omega_0$, $\omega_{1,\lm}$, $\omega_{2,\lm}$, $a_\lm$ and $b_\lm$ in \eqref{QuasimdePsi} can be chosen so that $\psi_\lm$ will belong to $\dmn H_\lm$. First of all, the $L_2(\Real)$-norm of $\psi_\lm$ is finite if and only if $\omega_\lm<0$; therefore we must impose the conditions $\omega_0<0$ (the case $\omega_0=0$ will be treated in Theorem~\ref{Theorem2}). Note
that $u$ and $v_k$ belong to the Sobolev space $W_2^2(\cI)$ as solutions of the equation $-y''+Uy=f$ with $f\in L_2(\cI)$.
By construction, $\psi_\lm$ and its first derivative are continuous at $x=-\ell$, then it is enough to ensure  the continuous differentiability of $\psi_\lm$ at $x=\ell$.

Since $\psi_\lm(\ell+0)-\psi_\lm(\ell-0)=\theta+\lm v_1(\ell)+\lm^2v_{2,\lm}(\ell)+\lm^3v_{3,\lm}(\ell)-a_{\lm}$,
we set
\begin{equation}\label{Alm}
a_{\lm}=\theta+\lm v_1(\ell)+\lm^2v_{2,\lm}(\ell)+\lm^3v_{3,\lm}(\ell).
\end{equation}
 To see this, we calculate
\begin{multline*}
 \psi'_\lm(\ell+0)-\psi'_\lm(\ell-0)=
 \omega_\lm a_{\lm}+b_{\lm}\rho'(0)
 -\lm v'_1(\ell)-\lm^2 v'_{2,\lm}(\ell)-\lm^3 v'_{3,\lm}(\ell)
 \\
 =\lm(\omega_0+\lm\omega_{1,\lm}+\lm^2\omega_{2,\lm})\left(\theta+\lm v_1(\ell)+\lm^2v_{2,\lm}(\ell)+\lm^3v_{3,\lm}(\ell)\right)
 \\
 +b_{\lm}
 -\lm v'_1(\ell)-\lm^2 v'_{2,\lm}(\ell)-\lm^3 v'_{3,\lm}(\ell)
 \\
 =\lm  \big(\omega_0\theta-v'_1(\ell)\big)
 +\lm^2\left(\omega_{1,\lm}\theta +\omega_0 v_1(\ell)-v'_{2,\lm}(\ell)
  \right)
  \\
  +\lm^3\left(\omega_{2,\lm}\theta+\omega_{1,\lm}v_1(\ell)+\omega_0 v_{2,\lm}(\ell)-v'_{3,\lm}(\ell)
  \right) +b_{\lm}
  \\
+\lm^4\Big(\omega_0v_{3,\lm}(\ell)+\omega_{1,\lm}(v_{2,\lm}(\ell)+\lm v_{3,\lm}(\ell))
\\
+\omega_{2,\lm}(v_1(\ell)+\lm v_{2,\lm}(\ell)+\lm^2v_{3,\lm}(\ell))\Big).
\end{multline*}
In order to achieve $\psi'_\lm(\ell+0)=\psi'_\lm(\ell-0)$, we assume
\begin{gather}\label{CondOmega01}
\omega_0=\theta^{-1} v'_1(\ell),\qquad
\omega_{1,\lm}=\theta^{-1}(v'_{2,\lm}(\ell)-\omega_0 v_1(\ell)),
\\\label{CondOmega2}
\omega_{2,\lm} = \theta^{-1}(v'_{3,\lm}(\ell)-
  \omega_0 v_{2,\lm}(\ell)-\omega_{1,\lm}v_1(\ell))
\\\label{Blm}
b_{\lm}=-\lm^4\big(\omega_0v_{3,\lm}+\omega_{1,\lm}(v_{2,\lm}+\lm v_{3,\lm})+ \omega_{2,\lm}(v_1+\lm v_{2,\lm}+\lm^2v_{3,\lm})\big)\vert_{x=\ell}.
\end{gather}
On the other hand, applying Proposition~\ref{PropVpEll} to problems \eqref{CPv1}--\eqref{CPv3}, we deduce
\begin{gather}\label{V1pProp}
\theta v'_1(\ell)=-\omega_0+ \int_{-\ell}^\ell V  u^2\,dx,\qquad
\theta v'_{3,\lm}(\ell)=-\omega_{2,\lm}+\int_{-\ell}^\ell f_{3,\lm} u\,dx,
\\\label{V2pProp}
  \theta v'_{2,\lm}(\ell)=-\omega_{1,\lm}+\int_{-\ell}^\ell V  v_1 u\,dx
  +\int_{-\ell}^\ell
 \big(V_1+\omega_0^2+g_\lm \big)u^2\,dx.
\end{gather}
Then combining \eqref{CondOmega01}, \eqref{CondOmega2}, \eqref{V1pProp}  and
\eqref{V2pProp} yields
\begin{gather}\label{Omega0pr}
\omega_0=\frac{1}{\theta^2+1}\int_{-\ell}^\ell V u^2\,dx,
   \\\nonumber
  \omega_{1,\lm}=
  \frac{1}{\theta^2+1}\Big(\int_{-\ell}^\ell  V v_1u\,dx-\theta\omega_0 v_1(\ell)
    +\int_{-\ell}^\ell
 \big(V_1+\omega_0^2+g_\lm \big)u^2\,dx\Big),
\\\nonumber
  \omega_{2,\lm}=
  \frac{1}{\theta^2+1}\Big(\int_{-\ell}^\ell f_{3,\lm} u\,dx-\theta(\omega_0 v_{2,\lm}(\ell)+\omega_{1,\lm}v_1(\ell))\Big).
\end{gather}

Since $V$ has a compact support, $\omega_0$ does not depend on $\ell$ and can be finally written in the form
\begin{equation}\label{omega0inProof}
  \omega_0=\frac{1}{\theta^2+1}\int_{\Real}V u^2\,dx.
\end{equation}
Moreover $\omega_0$ is negative if  condition \eqref{Omega0isPositive} holds; then $\omega_\lm$ is negative for all $\lambda$ small enough and therefore $\psi_\lm\in L_2(\Real)$.

The function $g_\lm$ has an infinitely small $L_2(\Real)$-norm as $\lambda\to 0$, since
 \begin{equation*}
 \|V_\lm-V-\lm V_1\|=o(\lm)\quad \text{as } \lm\to 0.
\end{equation*}
Consequently there exists limit $\omega_1=\lim_{\lm\to 0}\omega_{1,\lm}$, where
\begin{equation}\label{Omega1pr}
  \omega_1=
  \frac{1}{\theta^2+1}\left(\int_{-\ell}^\ell  V v_1u\,dx-\theta\omega_0 v_1(\ell)
    +\int_{-\ell}^\ell
 \big(V_1+\omega_0^2 \big)u^2\,dx\right).
\end{equation}
 But it is not obvious that  $\omega_1$ does not depend on $\ell$, because the right hand side of \eqref{Omega1pr} contains the integrand $\omega_0^2u^2$ without a compact support as well as the solution $v_1$ of \eqref{CPv1} which depends on $\ell$.
We first note that $u^2-\Theta^2$ is a function of compact support. Then we have
  \begin{equation}\label{FightEll1}
    \int_{-\ell}^\ell u^2\,dx=\int_{-\ell}^\ell(u^2-\Theta^2)\,dx+\int_{-\ell}^\ell\Theta^2\,dx
    =\int_{\Real}(u^2-\Theta^2)\,dx+\ell(\theta^2+1).
  \end{equation}
Next,  $v_1$ can be written as $v_1= v_*-\omega_0(u_1+\ell u)$, where $v_*$ is the solution  of the Cauchy problem $-v_*''+Uv_*=- V u$, $v_*(-\ell)=0$, $v'_*(-\ell)=0$.
 Invoking \eqref{omega0inProof}, we derive
\begin{multline}\label{FightEll2}
  \int_{-\ell}^\ell V v_1u\,dx=
  \int_{-\ell}^\ell V ( v_*-\omega_0u_1)u\,dx
  - \omega_0\ell \int_{-\ell}^\ell V  u^2\,dx\\
  = \int_{\Real} V ( v_*-\omega_0u_1)u\,dx-
  \omega_0^2\ell(\theta^2+1).
\end{multline}
In order to compute $v_1(\ell)$, we multiply the equation in \eqref{CPv1} by $u_1$ and integrate by parts twice
$(v'_1 u_1-v_1u_1')\big|_{-\ell}^\ell=\int_{-\ell}^\ell  V u u_1\,dx$.
Since $u_1(-\ell)=-\ell$, $u_1(\ell)=\ell \theta^{-1}+\theta_1$, $u_1'(\ell)=\theta^{-1}$ and $v'_1(\ell)=\omega_0\theta$, we obtain
\begin{equation}\label{FightEll3}
  v_1(\ell)=\omega_0\theta^2\theta_1
  -\theta\int_{\Real} V u_1 u \,dx.
\end{equation}
Substitute \eqref{FightEll1}--\eqref{FightEll3} into \eqref{Omega1pr}, to find
\begin{equation}\label{omega1inProof}
  \omega_1=
  \frac{1}{\theta^2+1}
  \bigg(
  \int_{\Real} V v_0 u\,dx+
  \int_{\Real} V_1 u^2\,dx
    +\omega_0^2 \int_{\Real}(u^2-\Theta^2)\,dx-\omega_0^2\theta^{3}\theta_1
\bigg),
\end{equation}
where $ v_0= v_*+\omega_0(\theta^2-1)u_1$.
Hence $\omega_1$ does not depend on $\ell$ either. A similar arguments can be applied to $\omega_{2,\lm}$. Therefore $\psi_\lm$ belongs to $\dmn H_\lm$ by our choice of $\omega_0$, $\omega_{1,\lm}$, $\omega_{2,\lm}$, $a_\lm$ and $b_\lm$.

\begin{prop}\label{NormPsiEst}
  There exist constants $c$ and $C$ such that
  \begin{equation*}
    c \omega_\lm^{-1/2}\leq \|\psi_\lm\|\leq C \omega_\lm^{-1/2}.
  \end{equation*}
\end{prop}
\begin{proof}
  We first note that the solutions $v_{2,\lm}$ and $v_{3,\lm}$ are bounded in $L_2(\Real)$ uniformly on $\lm$. In addition, by Proposition~\ref{PropVpEll} we have
  \begin{gather*}
     \|v_{2,\lm}\|_{C^1(\cI)}\leq C(|\omega_{1,\lm}|+\|V v_1+(V_1+g_\lm )u\|_{L_2(\cI)})\leq c_1,\\
     \|v_{3,\lm}\|_{C^1(\cI)}\leq C(|\omega_{2,\lm}|+\|f_{3,\lm}\|_{L_2(\cI)})\leq c_2,
  \end{gather*}
where $c_1$ and $c_2$ are independent of $\lm$. Combining these bounds with \eqref{Alm} and \eqref{Blm} yields
\begin{equation}\label{EstAlmBlm}
|a_\lm|\leq c_3,\qquad |b_\lm|\leq c_3\lm^4.
\end{equation}
 Therefore
the main contribution as $\lm\to 0$  to the norm of $\psi_\lm$ is given by the exponents $e^{\pm\omega_\lm(x\mp\ell)}$. A direct calculation verifies
  $\|e^{-\omega_\lm(x+\ell)}\|_{L_2(-\infty,-\ell)}=(2\omega_\lm)^{-1/2}$ and
  $\|e^{\omega_\lm(x-\ell)}\|_{L_2(\ell, +\infty)}=(2\omega_\lm)^{-1/2}$.
Hence $ \|\psi_\lm\|\sim a \omega_\lm^{-1/2}$ as $\lm\to 0$. In particular,
$\|\psi_\lm\|\sim a_0\lm^{-1/2}$ if $\omega_0\neq 0$ and $\|\psi_\lm\|\sim a_1\lm^{-1}$ if $\omega_0=0$.
\end{proof}

\begin{lem}\label{lemQMA}
The pair $(-\omega_\lm^2 , \phi_\lm)$ is a quasimode of $H_\lm$ with
the accuracy $o(\lm^{9/2})$ as $\lm\to 0$.
\end{lem}
\begin{proof}
  Let $r_\lm=(H_\lm+\omega^2_\lm I)\psi_\lm$.  Then  $(H_\lm+\omega^2_\lm I)\phi_\lm =\|\psi_\lm\|^{-1}r_\lm$. We must estimate the $L_2$-norm of $r_\lm$. Since  $e^{\pm\omega_\lm(x\mp \ell)}$ are exact solutions of $-\psi''+\omega^2_\lm\psi=0$ and $\supp \rho=[0,1]$, we have
\begin{equation}\label{RlmTh2}
 r_\lm(x)=
    -b_{\lm}(\rho''(x-\ell)-\omega^2_\lm\rho(x-\ell))\quad\text{for }\ell\leq x\leq \ell+1
\end{equation}
and $r_\lm(x)=0$ for other $x$ from set $\{x\colon |x|>\ell\}$. In
view of \eqref{EstAlmBlm}, we have the bound
\begin{equation}\label{RlmEstOut}
  |r_\lm(x)|\leq c_1\lm^4\qquad \text{for } |x|\geq \ell,
\end{equation}
because $\rho$ and $\rho''$ are bounded on $[0,1]$.
 Next, we calculate $r_\lm$ for $|x|<\ell$. Recalling \eqref{EqnHBSu} and \eqref{CPv1}--\eqref{CPv3}, we derive
\begin{multline*}
r_\lm=\left(-\tfrac{d^2}{dx^2}+U+\lm V_\lm+\omega^2_\lm\right) \psi_\lm
   \\
   =\left(-\tfrac{d^2}{dx^2}+U+\lm  V+
    \lm^2 V_1+\lm^2 g_\lm+\omega^2_\lm\right)
   (u+\lm v_1+\lm^2 v_2+\lm^3 v_3)
   \\
   = (-u''+Uu)+\lm (-v_1''+Uv_1+ V u)
   +\lm^2 \big(-v_2''+Uv_2+ V v_1
   \\
   + V_1u+\omega_0^2u+ g_\lm u\big)
   +\lm^3 \big(-v_3''+Uv_3+f_{3,\lm}\big)+\lm^4 R_\lm=\lm^4 R_\lm,
\end{multline*}
where the norm $\|R_\lm\|_{L_2(\cI)}$ is bounded uniformly with respect to $\lm$.
From this we conclude that $\|r_\lm\|_{L_2(\cI)}=O(\lm^4)$, and
hence that $\|r_\lm\|=O(\lm^4)$ as $\lm\to 0$, in view of
\eqref{RlmEstOut}. Finally we have
\begin{equation}\label{EstQM}
  \|(H_\lm+\omega^2_\lm I)\phi_\lm\| =\|\psi_\lm\|^{-1}\|r_\lm\|
  \leq c \lm^{9/2}
\end{equation}
as $\lm\to 0$, by Proposition~\ref{NormPsiEst}.
\end{proof}

Owing to Lemmas~\ref{LemQuasimodes} and \ref{lemQMA}, the operator
$H_\lm$ possesses a negative eigenvalue $e_\lm$ satisfying the bound
$
|e_\lm+\lm^2(\omega_0+\lm\omega_{1,\lm}+\lm^2\omega_{2,\lm})^2|\leq
c\lm^{9/2}$. Since
\begin{equation*}
  (\omega_0+\lm\omega_{1,\lm}+\lm^2\omega_{2,\lm})^2
  -(\omega_0+\lm\omega_1)^2\sim 2\omega_0\lm(\omega_{1,\lm}-\omega_1)
\end{equation*}
and $\omega_{1,\lm}-\omega_1=o(1)$ as $\lm\to 0$, we derive the
asymptotic formula
\begin{equation*}
  e_\lm+\lm^2(\omega_0+\lm\omega_{1})^2=o(\lm^{3}),
\end{equation*}
which we rewrite as $(\lm^{-1}
\sqrt{-e_\lm})^2-(\omega_0+\lm\omega_{1})^2=o(\lm)$. From this we
immediately deduce that $\lm^{-1}
\sqrt{-e_\lm}+\omega_0+\lm\omega_{1}=o(\lm)$, and hence that
\begin{equation*}
 \sqrt{-e_\lm}=-\lm(\omega_0+\lm\omega_{1}+o(\lm))
\end{equation*}
as $\lm\to 0$, which completes the proof of Theorem~\ref{Theorem1}.

\subsection{Proof of Theorem~\ref{Theorem2}}
Now we consider the critical case when  inequality \eqref{Omega0isPositive} turns into the equality
\begin{equation}\label{IntVu20}
  \int_\Real V u^2\,dx=0.
\end{equation}
Hence $\omega_0=0$  in view of \eqref{Omega0pr}. Therefore
 $v_0= v_*$ and  \eqref{omega1inProof} becomes
\begin{equation*}
  \omega_1=
  \frac{1}{\theta^2+1}
  \bigg(
  \int_{\Real}V v_* u\,dx
  +
  \int_{\Real} V_1u^2\,dx
\bigg).
\end{equation*}

We must prove that a negative eigenvalue of $H_\lm$ exists, the key point being that this existence assertion follows from  inequality $\omega_1<0$. Indeed, almost eigenfunction $\psi_\lm$ given by \eqref{QuasimdePsi} belongs to $L_2(\Real)$ if  $\omega_\lm=\lm^2(\omega_{1,\lm}+\lm\omega_{2,\lm})$ is
negative, at least for small $\lm$. Condition \eqref{betaIntVdiff} ensures  $\omega_{1,\lm}<0$ and thereby $\omega_\lm<0$ for $\lambda$ small enough.

By Proposition~\ref{NormPsiEst}, we have $\|\psi_\lm\|\sim a\lm^{-1}$ as $\lm\to 0$, provided $\omega_0=0$. Therefore estimate \eqref{EstQM} can be improved $
\|(H_\lm+\omega^2_\lm)\phi_\lm\| \leq c \lm^{5}$ and then
\begin{equation*}
 |e_\lm+\lm^4(\omega_{1,\lm}+\lm\omega_{2,\lm})^2|\leq c_1\lm^{5}.
\end{equation*}
As in the proof of Theorem~\ref{Theorem1}, if we rewrite this bound in the form
\begin{equation*}
 (\lm^{-2} \sqrt{-e_\lm})^2-(\omega_{1,\lm}+\lm\omega_{2,\lm})^2= O(\lm),
\end{equation*}
then we derive $\lm^{-2}
\sqrt{-e_\lm}=-\omega_{1,\lm}+O(\lm)=-\omega_1+o(1)$. Finally, we have
\begin{equation*}
 \sqrt{-e_\lm}=-\lm^{2}(\omega_1+o(1))\quad\text{as } \lm\to 0,
\end{equation*}
and this is precisely the assertion of Theorem~\ref{Theorem2}.

\subsection{Proof of Corollary \ref{Cor3}}
This statement differs from all earlier proved by the fact that here the threshold eigenvalue exists for both positive and negative $\lambda$ small enough.
For the case $V_\lm=V$ this result has been proved by Simon
\cite{Simon:1976}.

Since $\int_\Real V\,dx=0$, from \eqref{Omega0Omega1ForU0} we observe $\omega_0=0$ and
\begin{equation*}
  \omega_1=\frac{1}{4} \iint_{\Real^2}V(x)\,|x-y|\,V(y)\,dx\,dy.
\end{equation*}

\begin{prop}
If $V$ is a function of zero mean, then
 \begin{equation*}
  \iint_{\Real^2}V(x)\,|x-y|\,V(y)\,dx\,dy=-2\int_{\Real} \left(
   \int_{-\infty}^x V(y)\,dy\right)^2 dx.
\end{equation*}
\end{prop}

\begin{proof}
From $\int_\Real V\,dx=0$ we immediately deduce
  \begin{gather*}
    \int_{-\infty}^x  V(y)\,dy=-\int^{+\infty}_x  V(y)\,dy, \\
    \int_{\Real} V(x)\int_{-\infty}^x y V(y)\,dy\,dx=-\int_{\Real} V(x)\int^{+\infty}_x
    y V(y)\,dy\,dx.
  \end{gather*}
Therefore
   \begin{multline*}
    \iint_{\Real^2}V(x)\,|x-y|\,V(y)\,dx\,dy
    =\int_{\Real} V(x)\int_{-\infty}^x (x-y) V(y)\,dy\, dx\\
    +
     \int_{\Real} V(x)\int^{+\infty}_x(y-x)  V(y)\,dy\, dx
     =
     2\int_{\Real} xV(x)\int_{-\infty}^xV(y)\,dy\, dx\\-
     2\int_{\Real} V(x)\int_{-\infty}^xyV(y)\,dy\, dx
     =
     4 \int_{\Real} xV(x)\int_{-\infty}^xV(y)\,dy\, dx,
  \end{multline*}
because integrating by parts yields
\begin{equation*}
 \int_{\Real} V(x)\int_{-\infty}^xyV(y)\,dy dx=-\int_{\Real} xV(x)\int_{-\infty}^xV(y)\,dy\, dx.
\end{equation*}
The proof is completed by showing that
   \begin{equation*}
     \int_{\Real} xV(x)\int_{-\infty}^xV(y)\,dy\, dx=-\frac12\int_{\Real} \left(
   \int_{-\infty}^x V(y)\,dy\right)^2 dx.
  \end{equation*}
\end{proof}

In view of this proposition, if potential $V$ is different from zero, then $\omega_1<0$. Hence $\omega_\lm=\lm^2\omega_{1,\lm}+\lm^3\omega_{2,\lm}$ is negative for $\lambda$ small enough, positive or negative.

\medskip
\emph{Acknowledgements.}
The author is indebted to the unknown Referee for a careful reading of this paper. The first version of the paper was significantly improved by the suggestions of the Referee.


\end{document}